\documentclass[11pt]{amsart}
\usepackage{amssymb}
\usepackage{amsfonts}
\usepackage{amsmath}
\usepackage{mathrsfs}
\usepackage{graphicx}
\usepackage{color}
\usepackage{amsmath}
\usepackage{stmaryrd}
\usepackage{epsfig,color}
\usepackage{blindtext}
\usepackage{enumerate}
\usepackage{hyperref}
\usepackage{url}
\usepackage{bbm}
\usepackage{filecontents}
\usepackage{nicefrac,mathtools}
\DeclareGraphicsExtensions{.pdf,.jpeg,.png}
\usepackage{epstopdf}
\usepackage{cancel} %for editing
\usepackage[normalem]{ulem} %for editing
\usepackage{verbatim}		%for comment

\usepackage{color}
\usepackage[msc-links, lite]{amsrefs}
\usepackage{geometry}
\newtheorem{theorem}{Theorem}[section]

\newtheorem{proposition}[theorem]{Proposition}
\newtheorem{lemma}[theorem]{Lemma}

\newtheorem{question}[theorem]{Question}

\newtheorem{claim}[]{Claim}

\theoremstyle{definition}
\newtheorem{definition}[theorem]{Definition}
\theoremstyle{remark}
\newtheorem{remark}[theorem]{Remark}

\def\wti{\widetilde}
\def\mb{\mathbb}

\def\area{\mathrm{Area}}

\def\vol{\mathrm{Vol}}

\newcommand{\mc}{\mathcal}
\newcommand{\mf}{\mathbf}
\newcommand{\spt}{\mathrm{spt\,}}

\numberwithin{equation}{section}

\title[Minimal hypersurfaces with non-empty free boundary]{Existence of minimal hypersurfaces with non-empty free boundary for generic metrics}
\date{\today}

\author{Zhichao Wang}
\address{Max-Planck Institute for Mathematics, Vivatsgasse 7, 
53111 Bonn, Germany}
\email{wangzhichaonk@gmail.com}

\begin{document}

\begin{abstract}
For almost all Riemannian metrics (in the $C^\infty$ Baire sense) on a compact manifold with boundary $(M^{n+1},\partial M)$, $3\leq (n + 1)\leq 7$, we prove that, for any open subset $V$ of $\partial M$, there exists a compact, properly embedded free boundary minimal hypersurface intersecting $V$.
\end{abstract}
\maketitle

\section{Introduction}
In 1960s, Almgren \citelist{\cite{Alm62}\cite{Alm65}} initiated a variational theory to find minimal submanifolds in any compact manifolds with boundary. For a closed manifold $M^{n+1}$, the regularity of such hypersurfaces was improved by Pitts \cite{Pit76} for $n\leq 5$, and Schoen-Simon \cite{SS} for $n = 6$. Very recently, Li and Zhou finished this program for a general compact manifold with nonempty boundary in \cite{LZ16}, in which they proved that every compact manifold with boundary admits a smooth compact minimal hypersurface with (possibly empty) free boundary. This result left widely open the following well-known question::
\begin{question}
Does every compact manifold with non-empty boundary admit a minimal hypersurface with non-empty free boundary? 
\end{question}

We point out that there are similar questions in any free boundary variational theory. In particular,  in the mapping approach by Fraser \cite{Fra00}, Lin-Sun-Zhou \cite{LSZ18}, and Lauren-Petradis \cite{LP18}, it was not known whether their free boundary minimal surfaces have nontrivial boundary.

In this paper, we solve this problem in generic scenarios and prove a much stronger property: $M$ admits infinitely many embedded minimal hypersurfaces with non-empty free boundary.
\begin{theorem}\label{thm:intro:main thm}
Let $(M^{n+1},\partial M)$ be a compact manifold of dimension $3\leq (n + 1)\leq 7$. Then for a $C^\infty$-generic Riemannian metric $g$ on $(M,\partial M)$, the union of boundaries of all smooth, embedded, free boundary minimal hypersurfaces is dense in $\partial M$.
\end{theorem}

We remark that a compact manifold with non-negative Ricci curvature and strictly convex boundary has no closed minimal hypersurface by \cite{FL}*{Lemma 2.2}. Therefore, by Marques-Neves \cite{MN17} and Li-Zhou \cite{LZ16}, it is known that there exist infinitely many properly embedded free boundary minimal hypersurfaces in such ambient manifolds.

For a generic metric on $(M,\partial M)$, the author together with Guang, Li and Zhou proved the density of free boundary minimal hypersurfaces in \cite{GLWZ19}*{Theorem 1.3}. Making use of a maximum principle by White \cite{Whi10}, such denseness gives that $M$ contains minimal hypersurfaces with non-empty boundary by merely assuming strict mean convexity at one point of the boundary $\partial M$ for a generic metric; see \cite{GLWZ19}. However, without any topological or curvature assumptions, it is in general very difficult to prevent the free boundary components from degenerating in the limit process (see e.g. \citelist{\cite{ACS17}\cite{GWZ18_2}}). Our theorem in this paper greatly improves this result by dropping off mean convexity assumption at one point.

\vspace{0.5em}
The denseness result in \cite{GLWZ19}*{Theorem 1.3} can be seen as a natural free boundary analog of \cite{IMN17}. The key ingredient of \cite{IMN17} by Irie, Marques and Neves is the Weyl law for the volume spectrum proved by Liokumovich, Marques and Neves in \cite{LMN16}. The volume spectrum of a compact Riemannian manifold with boundary $(M^{n+1}, g)$ is a nondecreasing sequence of numbers $\{\omega_k(M; g) : k\in\mb N\}$ defined variationally by performing a min-max procedure for the area functional over multiparameter sweepouts. The first estimates for these numbers were proven by Gromov \cite{Gro03} in the late 1980s (see also \cite{Gu09}). A direct corollary of Weyl Law they used is that, for $k$ large enough, $\omega_k(M;g)\neq \omega_k(M;g')$ whenever $\vol(M,g)\neq \vol(M,g')$.

Another observation by Irie, Marques and Neves is that such spectrum depends continuously on the metrics of $M$; see \cite{IMN17}*{Lemma 2.1} and \cite{MNS19}*{Lemma 1}. Applying this, they showed that continuous perturbations in an open set must create new minimal hypersurfaces intersecting that set.

\vspace{0.5em}
In this paper, we also borrow the idea from Irie-Marques-Neves \cite{IMN17}. However, the original perturbation would only produce new free boundary minimal hypersurfaces intersecting \emph{an open set}, but not an $n$-dimensional subset, that we need to consider here. To overcome this new challenge, we perturb the metric $g$ around a boundary point in \emph{a special way} so that a hypersurface whose boundary does not intersect the prescribed subset of $\partial M$ can also be regarded a hypersurface in $(M,\partial M,g)$. Recall that Weyl law in \cite{LMN16} gives that for large $k$, $\omega_k$ will change continuously if the volume of $M$ is changed under the perturbation. From these two observations, we are able to prove that such a special perturbation would produce new minimal hypersurfaces with free boundary intersecting the prescribed subset of $\partial M$.

\vspace{0.5em}
We finish the introduction with the idea of the construction of the special perturbation. Making use of the cut-off trick, the unit inward normal vector field of $\partial M$ can be extended to the whole $M$. Also, by multiplying another cut-off function, we can always construct a vector field whose support is close to our prescribed open set of $\partial M$. Such a vector field would give a one-parameter family of diffeomorphisms (not surjective) of $M$. Then the pull back metric given by such family is the desired perturbation since it is isometric to a subset of $M$ with the original metric. We refer to Proposition \ref{prop:free:dense} for more details.

\subsection*{Acknowledgement}
The author would like to thank Prof. Xin Zhou for bringing this problem to us and many helpful discussion.

\section{Preliminaries}
Let $(M^{n+1},g)$ be a smooth compact connected Riemannian manifold with nonempty boundary $\partial M$ and $3\leq (n+1)\leq 7$. Moreover, $M$ can always be embedded to a closed Riemannian manifold $\wti{M}$  which has the same dimension with $M$.  We can also assume that $\wti{M}$ is isometrically embedded in some $\mb{R}^L$ for $L$ large enough. 

\subsection{Geometric measure theory}
We now recall some basic notations in geometric measure theory; see \cite{LZ16}.

We use $\mc{V}_k(M)$ to denote the closure of the space of $k$-dimensional rectifiable varifolds in $\mb{R}^L$ with support contained in $M$. Let $\mc{R}_k(M;\mb Z_2)$ (resp. $\mc{R}_k(\partial M;\mb Z_2)$) be the space of $k$-dimensional modulo two flat chains of finite mass in $\mb{R}^L$ which are supported in $M$ (resp. in $\partial M$). Denote by $\spt T$ the support of $T\in\mc R_k(M;\mb Z_2)$. Given any $T\in \mc{R}_k(M;\mb Z_2)$, denote by $|T|$ and $\|T\|$ the integer rectifiable varifold and the Radon measure in $M$ associated with $T$, respectively. The mass norm and the flat metric on $\mc{R}_k(M;\mb Z_2)$ are denoted by $\mf{M}$ and $\mc{F}$ respectively; see \cite{Fed69}. Set 
\[ Z_k(M,\partial M;\mb Z_2)=\{T\in \mc{R}_k(M;\mb Z_2) : \spt(\partial T)\subset \partial M \}.\]
We say that $T,S\in Z_k(M,\partial M;\mb Z_2)$ are equivalent if $T-S\in \mc{R}_k(\partial M;\mb Z_2)$. We use $\mc{Z}_k(M,\partial M; \mb Z_2)$ to denote the space of all such equivalent classes; see \cite{GLWZ19}*{Section 3} for the equivalence with the formulation using integer rectifiable currents in \cite{LZ16} . 

The flat metric and the mass norm in the space of relative cycles are defined, respectively, as
\begin{gather*}
\mc F(\tau_1,\tau_2)=\inf\{\mc F(T):T\in\tau\},\ \ \ \mf M(\tau )=\inf\{\mf M(T):T\in\tau \}.
\end{gather*}

The connected component of $\mc Z_n(M,\partial M;\mb Z_2)$ containing 0 is weakly equivalent to $\mathbb{RP}^\infty$ by Almgren \cite{Alm62} (see also \cite{LMN16}*{\S 2.5} and \cite{GLWZ19}*{Section 3}). Denote by $\bar\lambda$ the generator of $H^1(\mc Z_n (M,\partial M ;\mb Z_2 );\mb Z_2)=\mb Z_2$.

\subsection{Auxiliary Lemmas}
In this part, we introduce some Lemmas in \citelist{\cite{IMN17}\cite{MNS19}\cite{GLWZ19}}.

Let $X$ be a finite dimensional simplicial complex. A continuous map $\Phi: X\rightarrow \mc Z_n (M,\partial M ; \mb Z_2)$ is called a \emph{$k$-sweepout} if
\[\Phi^*(\bar{\lambda}^k)\neq 0\in H^k(X;\mb Z_2).\]
We denote by $\mc P_k (M)$ the set of all $k$-sweepouts that have \emph{no concentration of mass}, meaning that 
\[\lim_{r\rightarrow 0}\sup\{\mf M(\Phi(x)\cap B_r(p)) : x\in X, p \in M \} = 0.\]

\begin{definition}
The {\em $k$-width of $(M,\partial M;g)$} is defined as
\[ \omega_k(M):=\inf_{\Phi\in\mc P_k (M)}\sup\{\mf M(\Phi(x)) :x\in  \mathrm{dmn}(\Phi)\},\]
where $\mathrm{dmn}(\Phi)$ is the domain of $\Phi$.
\end{definition}

For any compact Riemannian manifold with boundary $(M,\partial M,g)$, the sequence $\{\omega_p(M)\}$ satisfies Weyl Law, which is proven by Liokumovich, Marques and Neves.
\begin{theorem}[Weyl Law for the Volume Spectrum; \cite{LMN16}]
\label{thm:weyl}
There exists a constant $\alpha(n)$ such that, for every compact Riemannian manifold $(M^{n+1},\partial M,g)$ with (possibly empty) boundary, we have 
\[\lim_{k\to \infty}\omega_k(M;g)k^{-\frac{1}{n+1}}=\alpha(n)\vol(M,g)^{\frac{n}{n+1}}.\]
\end{theorem}

Irie-Marques-Neves \cite{IMN17}*{Lemma 2.1} proved that $\omega_k(M;g)$ depends continuously on metrics. The following is an improved version given by Marques, Neves and Song. 
\begin{lemma}[\citelist{\cite{IMN17}*{Lemma 2.1}\cite{MNS19}*{Lemma 1}}]\label{lem:wk continuous}
Let $g_0$ be a $C^2$ Riemannian metric on $(M,\partial M)$, and let $C_1 < C_2$ be positive constants. Then there exists $K=K( g_0,C_1,C_2) > 0$ such that
\[
\big|p^{-\frac{1}{n+1}}\omega_p (M;g) -p^{-\frac{1}{n+1}}\omega_p (M;g')\big|\leq K \cdot |g- g'|_{g_0} \]
for any $C^2$ metric $g, g' \in \{h ; C_1g_0\leq h \leq C_2 g_0\}$ and any $p\in\mb N$.
\end{lemma}

Inspired by Marques-Neves \cite{MN16}, the author with Guang, Li, and Zhou (see \cite{GLWZ19}*{Theorem 2.1}) gave a general index estimate for min-max minimal hypersurfaces with free boundary. Combining with a compactness theorem in \cite{GWZ18_2} by the author and Guang and Zhou, we also proved in \cite{GLWZ19} that the $k$-width is realized by the area (counting mul-
tiplicities) of min-max free boundary minimal hypersurfaces.
\begin{proposition}[\citelist{\cite{IMN17}*{Proposition 2.2}\cite{GLWZ19}*{Proposition 7.3}} ]\label{prop:wk can be realized}
Suppose $3\leq (n + 1)\leq 7$. Then for each $k\in\mb N$, there exist a finite disjoint collection $\{\Sigma_1,...,\Sigma_N\}$ of almost properly embedded free boundary minimal hypersurfaces in $(M,\partial M,g)$, and integers $\{m_1,...,m_N\}\subset\mb N$, such that
\[\omega_k(M;g)=\sum_{j=1}^Nm_j\area_g(\Sigma_j)\ \ \text{ and }\ \ \sum_{j=1}^N \mathrm{index}(\Sigma_j)\leq k.\]
\end{proposition}
\begin{remark}
In a recent exciting work \cite{Zhou19}, X. Zhou proved that, for a bumpy metric on a closed manifold, each $m_j$ equals to 1, which is conjectured by Marques and Neves in \cite{MN16}. Based on this Multiplicity One Theorem, Marques-Neves \cite{MN18} proved that the index is in fact equals to $k$ for min-max minimal hypersurfaces realizing $\omega_k$.
\end{remark}

\section{Proof of Theorem \ref{thm:intro:main thm}}
Let  $(M^{n+1},\partial M)$ be a compact manifold with boundary and $3\leq (n+1)\leq 7$. Let $\mc{M}$ be the space of all smooth Riemannian metrics on $M$, endowed with the smooth topology. Suppose that $V\subset\partial M$ is a non-empty open set.  Let $\mc{M}_V$ be the set of metrics $g\in \mc{M}$ such that there exists a non-degenerate, properly embedded free boundary minimal hypersurface $\Sigma$ in $(M,\partial M, g)$ whose boundary intersects $V$.

We approach the theorem by proving the following proposition.
\begin{proposition}\label{prop:free:dense}
For any compact manifold $(M,\partial M)$ and any open subset $V\subset \partial M$, $\mc{M}_V$ is open and dense in $\mc{M}$ in the smooth topology.
\end{proposition}
\begin{proof}
Let $g \in\mc M_V$ and $\Sigma$ be like in the statement of the proposition. Following the step by Irie-Marques-Neves in \cite{IMN17}, we first show the openness of $\mc M_V$. Note that $\Sigma$ is a properly embedded, then the Structure Theorem of White \cite{Whi91}*{Theorem 2.1} (see \cite{ACS17}*{Theorem 35} for a version on free boundary minimal hypersurfaces) also gives that for every Riemannian metric $g'$ sufficiently close to $g$, there exists a unique nondegenerate properly embedded free boundary minimal hypersurface $\Sigma'$ close to $\Sigma$. This implies $\mc M_V$ is open.

It remains to show the set $\mc M_V$ is dense. Let $g$ be an arbitrary smooth Riemannian metric on $(M,\partial M)$ and $\mc V$ be an arbitrary neighborhood of $g$ in the
$C^\infty$ topology. By the Bumpy Metrics Theorem (\citelist{\cite{Whi91}*{Theorem 2.1}\cite{ACS17}*{Theorem 9}}), there exists $g'\in\mc V$ such that every compact, almost properly embedded free boundary minimal hypersurface with respect to $g'$ is nondegenerate. 

Since $g'$ is bumpy, then by \cite{GLWZ19}*{Proposition 5.3} (see also \citelist{\cite{GWZ18_2}\cite{Wang19}}), the space of almost embedded free boundary minimal hypersurfaces with $\area\leq \Lambda$ and $\mathrm{index}\leq I$ is countable with respect to $g'$ for all $\Lambda>0$ and $I\geq 0$. Therefore, the set
\[
\mc C:=
\left\{\sum_{j=1}^Nm_j\area_{g'}(\Sigma_j)\,\Bigg|
\begin{array}{lll}
&N\in\mb N, \{m_j\}\subset \mb N,\{\Sigma_j\} \text{ disjoint collection of almost}\\
&\text{ properly embedded free boundary minimal }\\
&\text{ hypersurfaces in $(M,\partial M, g')$}
\end{array} 
\right\}
\]
is countable.

Let $U$ be an open set of $M$ such that $\overline U\cap \partial M\subset V$ is non-empty. Let $X$ be a vector field on $M$ so that $\spt X\subset U$ and for $x\in\partial M$ satisfying $X(x)\neq 0$, $X(x)/|X(x)|$ is the outward unit normal vector of $\partial M$. Denote by $(F_t)_{0\leq t\leq 1}$ a family of diffeomorphisms of $M$ generated by $X$. Set 
\[g_t=F_t^*g' \text{\ \ and\ \ } M_t=F_t(M).\]
Then $(M,\partial M,g_t)$ is isometric to $(M_t,\partial M_t,g')$. Note that we can take $\delta>0$ so that $g_t\in\mc V$ for all $t\in[0.\delta]$.
\begin{claim}\label{claim:mass always in C}
Let $\Gamma$ be an integral varifold in $M$ whose support is a free boundary minimal hypersurface $\Sigma$ (possibly disconnected) in $(M,\partial M,g_t)$. Assuming that $\partial\Sigma\,\cap V=\emptyset$, then $\mf M(\Gamma)\in\mc C$.
\end{claim}
\begin{proof}[Proof of Claim \ref{claim:mass always in C}]
By the definition of $g_t$, $\Sigma$ can be seen as a free boundary minimal hypersurface in $(M_t,\partial M_t,g')$ so that $\partial \Sigma\cap F_t(V)=\emptyset$. Thus, $\Sigma $ is also a free boundary minimal hypersurface in $(M,\partial M,g')$. It follows that $\mf M(\Gamma)\in \mc C$ (counted with multiplicities).
\end{proof}

\begin{claim}\label{claim:existence of almost proper}
There exist $t_1\in[0,\delta]$ and an almost properly embedded free boundary minimal hypersurface $(\Sigma_1,\partial \Sigma_1)\subset (M,\partial M,g_{t_1})$ satisfying $\partial \Sigma\cap V\neq \emptyset$.
\end{claim}
\begin{proof}[Proof of Claim \ref{claim:existence of almost proper}]
Suppose not, then for all $t\in[0,\delta]$, all the almost properly embedded minimal hypersurfaces in $(M_t,\partial M_t,g')$ have no boundaries in $V$. Recall that Proposition \ref{prop:wk can be realized} gives that $\omega_k(M;g_t)$ is realized by the area of such hypersurfaces. Together with Claim \ref{claim:mass always in C}, we conclude that
\[\omega_k(M_t;g')\in \mc C\ \ \text{ for all $t\in[0,\delta]$ and $k\in\mb N$. }\]
On the other hand, the Weyl law (see Theorem \ref{thm:weyl}) implies that $\omega_k(M;g_\delta)<\omega_k(M;g')$ for $k$ large enough. The Lemma \ref{lem:wk continuous} deduces that $\omega_k(M;g_t)$
is continuous, which leads to a contradiction with that $\omega_k(M;g_t)$ lies in a countable set. The proof is finished.
\end{proof}
Thus we have proved that for some $t_1\in[0,\delta]$, there exists an almost properly embedded free boundary minimal hypersurface $(\Sigma_1,\partial \Sigma_1)\subset (M,\partial M;g_{t_1})$ satisfying $\partial \Sigma_1\cap V\neq \emptyset$. Then by \cite{IMN17}*{Proposition 2.3} (see also \citelist{\cite{MNS19}*{Lemma 4}\cite{GLWZ19}*{Proposition 7.6}}), $g_{t_1}$ can be perturbed to $g''\in \mc V$ so that $(M,\partial M,g'')$ contains an almost properly embedded, non-degenerate, free boundary minimal hypersurfaces $\Sigma''$ whose boundary intersects $V$. Finally, \cite{GLWZ19}*{Proposition 7.7} would allow us to perturb $g''$ to  $\wti g\in\mc V$ and $\Sigma''$ is a properly embedded free boundary minimal hypersurface in $(M,\partial M,\wti g)$. This implies that $\wti g\in \mc M_V$ and we are done.
\end{proof}

Now we are ready to prove Theorem \ref{thm:intro:main thm}. The proof is the same with that of \cite{IMN17}*{Main theorem}.
\begin{proof}[Proof of Theorem \ref{thm:intro:main thm}]
Let $\{V_i\}$ be a countable basis of $\partial M$. Since, by Proposition \ref{prop:free:dense}, each $\mc M_{V_i}$ is open and dense in $\mc M$, and hence the set $\bigcap _i \mc M_{V_i}$ is $C^\infty$ Baire-generic in $\mc M$. This finishes the proof.
\end{proof}

\bibliographystyle{amsalpha}
%\bibliography{minmax}

\end{document}